\documentclass[a4paper, 10pt]{article}
\usepackage{mathrsfs}
\usepackage{amsmath}
\usepackage{amsfonts}
\usepackage[T1]{fontenc}
\usepackage[latin9]{inputenc}
\usepackage{amsthm}
\usepackage{graphicx}
\usepackage{amsmath}

\usepackage{amsthm}
\usepackage{array}
\usepackage{cases}
\makeatletter
\def\th@plain{%
  \upshape 
}
\makeatother

\makeatletter
\renewenvironment{proof}[1][\proofname]{\par
  \pushQED{\qed}%
  \normalfont \topsep6\p@\@plus6\p@\relax
  \trivlist
  \item[\hskip\labelsep
        \bfseries
    #1\@addpunct{.}]\ignorespaces
}{%
  \popQED\endtrivlist\@endpefalse
}
\makeatother

\numberwithin{equation}{section}

\newtheorem{thm}{Theorem}

\numberwithin{equation}{section}

\usepackage[varg]{txfonts}
\usepackage{graphicx, epsfig, subfigure}
\usepackage{enumerate} 
\usepackage[square, numbers, sort&compress]{natbib} 
\ifx\pdfoutput\undefined
 \usepackage[dvipdfm,%
  pdfstartview=FitH, 
  bookmarks=true,%
  bookmarksnumbered=true, 
  bookmarksopen=true, 
  plainpages=false,%
  pdfpagelabels,%
  colorlinks=true, 
  linkcolor=blue, 
  citecolor=blue,%
  urlcolor=black,
  pdfborder=001]{hyperref}
  \AtBeginDvi{}  
\else
 \usepackage[pdftex,%
  pdfstartview=FitH, 
  bookmarks=true,%
  bookmarksnumbered=true, 
  bookmarksopen=true, 
  plainpages=false,%
  pdfpagelabels,%
  colorlinks=true, 
  linkcolor=blue, 
  citecolor=blue,%
  urlcolor=black,
  pdfborder=001]{hyperref}
\fi

\numberwithin{equation}{section}

\begin{document}

\title{\LARGE A note on the edge partition of graphs containing either a light edge or an alternating 2-cycle\thanks{This is a very preliminary version! If you find any topes or
  mistakes, please fell free to let us now (contact at xzhang@xidian.edu.cn or beiniu@stu.xidian.edu.cn). This paper is used for
  communication, and will not be published as it is in a journal.}
}
\author{ Xin Zhang, ~~Bei Niu      \\
{\small School of Mathematics and Statistics, Xidian University, Xi'an, 710071, China}}
%
\maketitle
\begin{abstract}\baselineskip 0.46cm
Let $\mathcal{G}_{\alpha}$ be a hereditary graph class (i.e, every subgraph of $G_{\alpha}\in \mathcal{G}_{\alpha}$ belongs to $\mathcal{G}_{\alpha}$) such that every graph $G_{\alpha}$ in $\mathcal{G}_{\alpha}$ has minimum degree at most 1, or contains either an edge $uv$ such that $d_{G_{\alpha}}(u)+d_{G_{\alpha}}(v)\leq \alpha$ or a 2-alternating cycle. It is proved that every graph in $\mathcal{G}_{\alpha}$ ($\alpha\geq 5$) with maximum degree $\Delta$ can be edge-partitioned into two forests $F_1$, $F_2$ and a subgraph $H$ such that $\Delta(F_i)\leq \max\{2,\lceil\frac{\Delta-\alpha+6}{2}\rceil\}$ for $i=1,2$ and $\Delta(H)\leq \alpha-5$.
%
\end{abstract}
\baselineskip 0.46cm

An edge partition of an undirected graph $G$ is a partition of its set of edges into some disjoint subsets so that their union is the set of edges of $G$. For example, an edge $k$-coloring of a graph $G$ is an edge partition $E(G)=E(G_1)\cup E(G_2)\cup \ldots \cup E(G_k)$ such that each $E(G_i)$ with $1\leq i\leq k$ induces a match.

Sometimes, when considering the edge partition problem of a graph $G$, we are to search disjoint subsets $E(G_1),E(G_2), \ldots ,E(G_k)$ such that $E(G)=E(G_1)\cup E(G_2)\cup \ldots \cup E(G_k)$, but $E(G_i)$ and $E(G_j)$ with some $1\leq i<j\leq k$ may induce different subgraphs. For example, Wang, Hu and Wang \cite{WHW} showed that

\begin{thm}\cite[Wang, Hu, Wang]{WHW}\label{aaa}
Every planar graph $G$ with maximum degree $\Delta$ can be edge-partitioned into two forests $F_1$, $F_2$ and a subgraph $H$ such that $\Delta(F_i)\leq \max\{2,\lceil\frac{\Delta-9}{2}\rceil\}$ for $i=1,2$ and $\Delta(H)\leq 10$.
\end{thm}

Similar result was established by Wang \cite{W} for planar graph with more restrictions.

\begin{thm}\cite[Wang]{W}\label{bbb}
Every planar graph without 3-cycle normally adjacent to a 4-cycle and with maximum degree $\Delta$ can be edge-partitioned into two forests $F_1$, $F_2$ and a subgraph $H$ such that $\Delta(F_i)\leq \max\{2,\lceil\frac{\Delta-4}{2}\rceil\}$ for $i=1,2$ and $\Delta(H)\leq 5$.
\end{thm}

The proofs of Theorems \ref{aaa} and \ref{bbb} are both based on the structure of the considered graph. Precisely, the proof of Theorem \ref{aaa} uses the fact that every planar graph $G$ with minimum degree at least 2 contains either an edge $uv$ such that $d_G(u)+d_G(v)\leq 15$ or a 2-alternating cycle, that is, a cycle $v_{0}v_{1}~\cdots v_{2n-1}v_0$ such that $d_G(v_{0})=d_G(v_{2})=\cdots=d_G(v_{2n-2})=2$, and the proof of Theorem \ref{bbb} involves the fact that every planar graph $G$ without 3-cycle normally adjacent to a 4-cycle and  with minimum degree at least 2 contains either an edge $uv$ such that $d_G(u)+d_G(v)\leq 10$ or a 2-alternating cycle. Clearly, we can feel at this stage that the proofs of Theorems \ref{aaa} and \ref{bbb} may be similar, but they are  full-length-proved in different papers.

The aim of this note is to establish a common tool that can be used to prove similar theorems as Theorems \ref{aaa} and \ref{bbb}.

Let $\mathcal{G}_{\alpha}$ be a hereditary graph class (i.e, every subgraph of $G_{\alpha}\in \mathcal{G}_{\alpha}$ belongs to $\mathcal{G}_{\alpha}$) such that every graph $G_{\alpha}\in \mathcal{G}_{\alpha}$ has minimum degree at most 1, or contains either an edge $uv$ such that $d_{G_{\alpha}}(u)+d_{G_{\alpha}}(v)\leq \alpha$ or a 2-alternating cycle. What can we say about the graph class $\mathcal{G}_{\alpha}$? In the following, we conclude the following theorem that implies Theorems \ref{aaa} and \ref{bbb} (taking $\alpha=15$ and $\alpha=10$, respectively).

\begin{thm}\label{ccc}
Every graph in $\mathcal{G}_{\alpha}$ ($\alpha\geq 5$) with maximum degree $\Delta$ can be edge-partitioned into two forests $F_1$, $F_2$ and a subgraph $H$ such that $\Delta(F_i)\leq \max\{2,\lceil\frac{\Delta-\alpha+6}{2}\rceil\}$ for $i=1,2$ and $\Delta(H)\leq \alpha-5$.
\end{thm}

In the following, we prove a slightly stronger result than Theorem \ref{ccc}.

\begin{thm}\label{ddd}
Every graph $G$ in $\mathcal{G}_{\alpha}$ ($\alpha\geq 5$) can be edge-partitioned into two forests $F_1$, $F_2$ and a subgraph $H$ such that $d_{F_i}(v)\leq \max\{2,\lceil\frac{d_{G}(v)-\alpha+6}{2}\rceil\}$ for every $v\in F_i$ with $i=1,2$  and $\Delta(H)\leq \alpha-5$.
\end{thm}

\begin{proof}
Let $G$ be a minimum counterexample to Theorem \ref{ddd}. Clearly, $G$ is connected.
If $\Delta(G)\leq \alpha-5$, then take $H=G$ and $F_{1}=F_{2}=\emptyset$ and we are done. Hence we assume $\Delta(G)\geq \alpha+2$.

\textbf{Case 1.} \emph{$G$ contains a vertex $u$ of degree at most 1.}

Since $G$ is connected, $d_G(u)=1$. Let $uv\in E(G)$ and let $G'=G-u$. By the minimality of $G$, $G'$ has an edge-partition into two forests $F_{1}^{'},F^{'}_{2}$ and a subgraph $H'$ such that $\Delta(H')\leq \alpha-5$ and $d_{F'_{i}}(v)\leq \max\{2,\lceil\frac{d_{G'}(v)-\alpha+6}{2}\rceil\}$ for any $v\in V(F_i)$ with $i=1,2$.

If $d_{H'}(v)\leq \alpha-6$, then let $H=H'+uv, F_1=F'_1$ and $F_2=F'_2$. Since $d_H(u)=1$, $d_{H}(v)=d_{H'}(v)+1\leq \alpha-5$, and $d_H(w)=d_{H'}(w)\leq \Delta(H')\leq \alpha-5$ for any $w\in V(H')\setminus \{v\}$, we have $\Delta(H)\leq \alpha-5$. On the other hand, it is clear that for any $v\in F_i$ with $i=1,2$, $d_{F_{i}}(v)=d_{F'_{i}}(v)\leq \max\{2,\lceil\frac{d_{G'}(v)-\alpha+6}{2}\rceil\}\leq \max\{2,\lceil\frac{d_{G}(v)-\alpha+6}{2}\rceil\}$.

If $d_{H'}(v)\geq \alpha-5$, then $d_{H'}(v)=\alpha-5$ because $\Delta(H')\leq \alpha-5$. Assume, without loss of generality, that $d_{F'_1}(v)\leq d_{F'_2}(v)$.
Since $d_G(v)-1=d_{G'}(v)=d_{F'_1}(v)+d_{F'_2}(v)+d_{H'}(v)=d_{F'_1}(v)+d_{F'_2}(v)+\alpha-5\geq 2d_{F'_1}(v)+\alpha-5$, $d_{F'_1}(v)\leq \lfloor\frac{d_G(v)-\alpha+4}{2}\rfloor$.

Let $H=H', F_1=F'_1+uv$ and $F_2=F'_2$. Clearly, $\Delta(H)=\Delta(H')\leq \alpha-5$ and for any $v\in F_2$, $d_{F_{2}}(v)=d_{F'_{2}}(v)\leq \max\{2,\lceil\frac{d_{G'}(v)-\alpha+6}{2}\rceil\}\leq \max\{2,\lceil\frac{d_{G}(v)-\alpha+6}{2}\rceil\}$. Moreover, $F_1$ is a forest so that $d_{F_1}(u)=1$, $d_{F_1}(v)\leq d_{F'_1}(v)+1\leq \lfloor\frac{d_G(v)-\alpha+4}{2}\rfloor+1\leq \max\{2,\lceil\frac{d_G(v)-\alpha+6}{2}\rceil\}$ and $d_{F_1}(w)= d_{F'_1}(w) \leq \max\{2,\lceil\frac{d_{G'}(w)-\alpha+6}{2}\rceil\}\leq \max\{2,\lceil\frac{d_{G}(w)-\alpha+6}{2}\rceil\}$ for any $w\in V(F'_1)\setminus \{v\}$.

\textbf{Case 2.} \emph{$G$ contains an edge $uv\in E(G)$ with $d_{G}(u)+d_{G}(v)\leq \alpha$.}

Let $G'=G-uv$. By the minimality of $G$, $G'$ has an edge-partition into two forests $F_{1}^{'},F^{'}_{2}$ and a subgraph $H'$ such that $\Delta(H')\leq \alpha-5$ and $d_{F'_{i}}(v)\leq \max\{2,\lceil\frac{d_{G'}(v)-\alpha+6}{2}\rceil\}$ for any $v\in V(F_i)$ with $i=1,2$.

If $\max\{d_{H'}(u),d_{H'}(v)\}\leq \alpha-6$, then let $H=H'+uv, F_1=F'_1$ and $F_2=F'_2$. It is easy to see that $\Delta(H)\leq \alpha-6+1=\alpha-5$ and for any $v\in F_i$ with $i=1,2$, $d_{F_{i}}(v)=d_{F'_{i}}(v)\leq \max\{2,\lceil\frac{d_{G'}(v)-\alpha+6}{2}\rceil\}\leq \max\{2,\lceil\frac{d_{G}(v)-\alpha+6}{2}\rceil\}$.

If $\max\{d_{H'}(u),d_{H'}(v)\}\geq \alpha-5$, then assume, without loss of generality, that $d_{H'}(u)\leq d_{H'}(v)=\alpha-5$, since $\Delta(H')\leq \alpha-5$.
By (a), we assume that $d_G(u)\geq 2$.

If $d_G(u)=2$, then $u$ cannot be incident with both $F'_1$ and $F'_2$. Without loss of generality, assume that $d_{F'_1}(u)=0$ and $d_{F'_2}(u)\leq 1$. Since $d_G(v)\leq \alpha-2$, $\alpha-3\geq d_G(v)-1=d_{G'}(v)=d_{F'_1}(v)+d_{F'_2}(v)+d_{H'}(v)=d_{F'_1}(v)+d_{F'_2}(v)+\alpha-5$, which implies $d_{F'_1}(v)+d_{F'_2}(v)\leq 2$.

\begin{itemize}
  \item If $d_{F'_1}(v)\leq 1$, then let $H=H'$, $F_1=F'_1+uv$ and $F_2=F'_2$. Clearly, $\Delta(H)=\Delta(H')\leq \alpha-5$ and for any $v\in V(F_2)$, $d_{F_{2}}(v)=d_{F'_{2}}(v)\leq \max\{2,\lceil\frac{d_{G'}(v)-\alpha+6}{2}\rceil\}\leq \max\{2,\lceil\frac{d_{G}(v)-\alpha+6}{2}\rceil\}$. Moreover, $F_1$ is a forest so that $d_{F_1}(u)=1$, $d_{F_1}(v)\leq 2$ and $d_{F_1}(w)=d_{F'_1}(w)\leq \max\{2,\lceil\frac{d_{G'}(w)-\alpha+6}{2}\rceil\}\leq \max\{2,\lceil\frac{d_G(w)-\alpha+6}{2}\rceil\}$ for any $w\in V(F'_1)\setminus \{v\}$.
  \item If $d_{F'_1}(v)=2$, then $d_{F'_2}(v)=0$. Let $H=H'$, $F_1=F'_1$ and $F_2=F'_2+uv$. It is easy to see that $\Delta(H)=\Delta(H')\leq \alpha-6$ and $d_{F_{1}}(v)=d_{F'_{1}}(v)\leq \max\{2,\lceil\frac{d_{G'}(v)-\alpha+6}{2}\rceil\}\leq \max\{2,\lceil\frac{d_{G}(v)-\alpha+6}{2}\rceil\}$ for any $v\in V(F_1)$. Moreover, $F_2$ is a forest so that $d_{F_2}(v)=1$, $d_{F_2}(u)\leq 2$ and $d_{F_2}(w)=d_{F'_2}(w)\leq \max\{2,\lceil\frac{d_{G'}(w)-\alpha+6}{2}\rceil\}\leq \max\{2,\lceil\frac{d_G(w)-\alpha+6}{2}\rceil\}$.
\end{itemize}

If $d_G(u)=3$, then $d_G(v)\leq \alpha-3$, which implies $\alpha-4\geq d_G(v)-1=d_{G'}(v)=d_{F'_1}(v)+d_{F'_2}(v)+d_{H'}(v)=d_{F'_1}(v)+d_{F'_2}(v)+\alpha-5$. Hence $d_{F'_1}(v)+d_{F'_2}(v)\leq 1$. Without loss of generality, assume that $d_{F'_1}(v)=0$ and $d_{F'_2}(v)\leq 1$. Since $d_G(u)=3$, $d_{F'_1}(u)+d_{F'_2}(u)\leq 2$. By similar arguments as the above two paragraphs (just exchanging $u$ and $v$ among those words), we get the desired edge-partition.

If $d_G(u)\geq 4$, then $\alpha-5\geq \alpha-1-d_G(u)\geq d_G(v)-1=d_{G'}(v)=d_{F'_1}(v)+d_{F'_2}(v)+d_{H'}(v)=d_{F'_1}(v)+d_{F'_2}(v)+\alpha-5$, which implies $d_{F'_1}(v)=d_{F'_2}(v)=0$ and $d_G(u)=4$. Hence $d_{F'_1}(u)+d_{F'_2}(u)\leq 3$ and we assume, without loss of generality, that $d_{F'_1}(u)\leq 1$.
Let $H=H'$, $F_1=F'_1+uv$ and $F_2=F'_2$. Clearly, $\Delta(H)=\Delta(H')\leq \alpha-5$ and for any $v\in V(F_2)$, $d_{F_{2}}(v)=d_{F'_{2}}(v)\leq \max\{2,\lceil\frac{d_{G'}(v)-\alpha+6}{2}\rceil\}\leq \max\{2,\lceil\frac{d_{G}(v)-\alpha+6}{2}\rceil\}$. Moreover, $F_1$ is a forest so that $d_{F_1}(v)=1$, $d_{F_1}(u)\leq 2$ and $d_{F_1}(w)=d_{F'_1}(w)\leq \max\{2,\lceil\frac{d_{G'}(w)-\alpha+6}{2}\rceil\}\leq \max\{2,\lceil\frac{d_G(w)-\alpha+6}{2}\rceil\}$ for any $w\in V(F'_1)\setminus \{v\}$.

\textbf{Case 3.} \emph{$G$ contains a 2-alternating cycle $C=v_{0}v_{1}\cdots v_{2n-1}v_{0}$, where $n\geq2$, such that $d_{G}(v_{0})=d_{G}(v_{2})=\cdots=d_{G}(v_{2n-2})=2$.}

Let $G'=G-E(C)$. By the minimality of $G$, $G'$ has an edge-partition into two forests $F_{1}^{'},F^{'}_{2}$ and a subgraph $H'$ such that $\Delta(H')\leq \alpha-5$ and $d_{F'_{i}}(v)\leq \max\{2,\lceil\frac{d_{G'}(v)-\alpha+6}{2}\rceil\}$ for any $v\in V(F_i)$ with $i=1,2$.

Let $H=H'$, $F_1=F'_1\cup \{v_1v_2,v_3v_4,\ldots,v_{2n-1}v_0\}$ and $F_2=F'_2\cup \{v_0v_1,v_2v_3,\ldots,v_{2n-2}v_{2n-1}\}$.
Clearly, $\Delta(H)=\Delta(H')\leq \alpha-5$ and $F_1,F_2$ are forests. For each integer $i\in\{1,3,\cdots,2n-1\}$, $d_G(v_i)+d_G(v_{i+1})=d_G(v_i)+2\geq \alpha+1$ (otherwise we come bake to Case 2), which implies that $d_{G'}(v_i)=d_G(v_i)-2\geq \alpha-3$. Therefore, for integers $i\in\{1,3,\cdots,2n-1\}$ and $j\in \{1,2\}$, $d_{F'_j}(v_i) \leq \max\{2,\lceil\frac{d_{G'}(v_i)-\alpha+6}{2}\rceil\}=\lceil\frac{d_{G'}(v_i)-\alpha+6}{2}\rceil$,
implying that $d_{F_j}(v_i)\leq d_{F'_j}(v_i)+1 \leq \lceil\frac{d_{G'}(v_i)-\alpha+6}{2}\rceil+1=\lceil\frac{(d_{G'}(v_i)+2)-\alpha+6}{2}\rceil=\lceil\frac{d_{G}(v_i)-\alpha+6}{2}\rceil$. On the other hand,
it is easy to see that $d_{F_j}(v_i)\leq 1$ for integers $i\in\{0,2,\cdots,2n-2\}$ and $j\in \{1,2\}$, and $d_{F_j}(w)=d_{F'_j}(w)\leq \max\{2,\lceil\frac{d_{G'}(w)-\alpha+6}{2}\rceil\}\leq \max\{2,\lceil\frac{d_{G}(w)-\alpha+6}{2}\rceil\}$ for any $w\in V(F_j)\backslash V(C)$. Therefore, $d_{F_{i}}(v)\leq \max\{2,\lceil\frac{d_G(v)-\alpha+6}{2}\rceil\}$ for any $v\in V(F_i)$ with $i=1,2$ and $H\cup F_1\cup F_2$ is a desired edge-partition of $G$.
\end{proof}

\noindent \textbf{Remark.} Wang, Li, Hu and Wang \cite{WLHW} gave the proof of Theorem \ref{ddd} for $\alpha\geq 7$. Here we improve it to the case $\alpha\geq 5$. Actually,
the idea of the proof of Theorem \ref{ddd} comes from the proofs of \cite[Theorem 1]{WLHW}, \cite[Theorem 2]{W} and \cite[Theorem 5]{WHW}, but here we involve a  simplification, especially when we consider the third case. Therefore, we believe that it is an easier proof. 

\vspace{3mm} It is well-known that every $K_4$-minor-free graph $G$ has a vertex of degree at most 1, or contains either an edge $uv$ such that $d_{G}(u)+d_{G}(v)\leq 6$ or a 2-alternating cycle of length 4, see \cite[Lemma 2.1]{JMT}. Therefore, every $K_4$-minor-free graph belongs to the graph class $\mathcal{G}_6$, and thus the following theorem is an immediate corollary from Theorem \ref{ccc}.

\begin{thm}\label{eee}
Every $K_4$-minor-free graph with maximum degree $\Delta$ can be edge-partitioned into two forests $F_1$, $F_2$ and a subgraph $H$ such that $\Delta(F_i)\leq \max\{2,\lceil\frac{\Delta}{2}\rceil\}$ for $i=1,2$ and $\Delta(H)\leq 1$ (i.e., $H$ is a matching).
\end{thm}

Theorem \ref{eee} is in accordance with \cite[Theorem 18]{WWW}, which was proved by Wang, Wang and Wang using similar arguments as proving Theorems \ref{aaa} and \ref{bbb} in the references \cite{WHW} and \cite{W}, respectively.


\bibliographystyle{srtnumbered}
\bibliography{mybib}

\end{document}